\documentclass[a4paper,12pt]{amsart}
\usepackage{amsthm,amsmath,amssymb}
\usepackage{pdfsync}
\usepackage{paralist}
\date{1st December 2019} 
\date{11th December 2019} 
\date{13th December 2019} 
\date{21st December 2019} 
\date{23st December 2019} 
\date{26st December 2019} 
\date{28st December 2019} 
\date{30st December 2019} 
\date{31st December 2019} 
\date{2nd January 2020} 
\date{3rd January 2020} 
\date{6th January 2020} 
\date{8th January 2020} 
\date{9th January 2020} 
\date{14th January 2020} 
\date{19th March 2021} 
\date{20th March 2021} 
\date{22th March 2021} 
\date{28th March 2021} 
\date{1st April 2021} 
\date{1st April 2021} 
\date{28th June 2021} 
\allowdisplaybreaks[4]
\makeatletter
\def\claimname{Claim}
\def\preclaimword{}

\def\@postclaim[#1]{\bf #1}
\def\@claim[#1]{\preclaimword {\bf #1.} \@ifnextchar[{\@postclaim}{}}
{\par\vskip1.5ex\noindent\@ifnextchar[{\@claim}{\bf\claimname~}}%
{\unskip\nobreak\hfill\par\vskip1.5ex}
\def\fracinline#1/#2{\mbox{\raise0.5ex\hbox{\footnotesize$#1$}{\hskip-.1em$/$\hskip-.1em}\raise-0.5ex\hbox{\footnotesize$#2$}}}
\def\fracinlines#1/#2{\mbox{\raise0.25ex\hbox{\tiny$#1$}{\hskip-.1em$/$\hskip-.1em}\raise-0.25ex\hbox{\tiny$#2$}}}
\makeatother
\makeatletter
\def\proofname{Proof}
\def\preproofword{\it Proof of }
\def\postproofword{\it :~}
\def\@pf[#1]{\preproofword {\it #1} \postproofword~}
\newenvironment{Pf}%
{\par\noindent\@ifnextchar[{\par\vskip 6pt\noindent\@pf}{\it\proofname:~ }}%
{{\unskip\nobreak~\hfill\it\qedsymbol}\par\vskip1.2ex}
\makeatother
\theoremstyle{plain}
  \newtheorem{thm}{Theorem}[section]
  
  \newtheorem{prop}[thm]{Proposition}
  \newtheorem{cor}[thm]{Corollary}
  \newtheorem{conj}[thm]{Conjecture}
\theoremstyle{definition}

\theoremstyle{remark}

\def\homeo{\approx}

\def\real{\mathbb{R}}

\def\incl{\mathrm{in}}
\def\id{\mathrm{id}}
\def\emptyarg{}
\def\ad#1{\def\thisarg{#1}\mathrm{ad}\ifx\thisarg\emptyarg\else(#1)\fi}
\def\const#1{\def\thisarg{#1}\operatorname{\iota}\ifx\thisarg\emptyarg\else(\hskip1pt#1)\fi}
\def\open#1{\def\thisarg{#1}\ifx\thisarg\emptyarg{\Site{Open}}\else{\Site{C^{#1}\text{-}Open}}\fi}
\def\domain#1{\def\thisarg{#1}\ifx\thisarg\emptyarg{\Site{Domain}}\else{\Site{C^{#1}\text{-}Domain}}\fi}
\def\convex#1{\def\thisarg{#1}\ifx\thisarg\emptyarg{\Site{Convex}}\else{\Site{C^{#1}\text{-}Convex}}\fi}
\def\locconv#1{\def\thisarg{#1}\ifx\thisarg\emptyarg{\Site{LocConv}}\else{\Site{C^{#1}\text{-}LocConv}}\fi}
\def\polyhedron#1{\def\thisarg{#1}\ifx\thisarg\emptyarg{\Site{Polyhedron}}\else{\Site{C^{#1}\text{-}Polyhedron}}\fi}
\def\smoothology#1{\def\thisarg{#1}\ifx\thisarg\emptyarg{\Category{Diffeology}}\else{\Category{C^{#1}\text{-}Diffeology}}\fi}
\def\der#1by#2{\frac{\operatorname{\mathit{d}}\hspace{-0.1mm}#1}{\operatorname{\mathit{d}}\hspace{-0.1mm}#2}}
\def\nder#1by#2times#3{\frac{\operatorname{\mathit{d}}^{#3}\hspace{-0.2mm}#1}{\operatorname{\mathit{d}}\hspace{-0.4mm}{#2\,}^{#3}}}
\def\pder#1by#2{\frac{\operatorname{\partial}\hspace{-0.1mm}#1}{\operatorname{\partial}\hspace{-0.1mm}#2}}
\def\npder#1by#2times#3{\frac{\operatorname{{\partial\,}^{#3}}\hspace{-0.2mm}#1}{\operatorname{\partial}\hspace{-0.1mm}{#2}^{#3}}}
\def\diff#1{\ifx#1(\operatorname{\mathit{d}}\hspace{0.25mm}#1\else\operatorname{\mathit{d}}\hspace{-0.5mm}#1\fi}
\def\pdiff#1{\ifx#1(\operatorname{\partial}\hspace{0.25mm}#1\else\operatorname{\partial}\hspace{-0.25mm}#1\fi}
\def\mathbold#1{\text{\boldmath{$#1$}}}

\def\Im#1{\operatorname{Im}\hskip1.5pt(\hskip1pt#1)}

\makeatletter
\def\@supp[#1]{\operatorname{supp}\hskip1pt(\hskip.5pt#1)}
\def\supp{\@ifnextchar[{\@supp}{\operatorname{supp}\hskip.5pt}}
\makeatother
\def\Category#1{{\sf{#1}}}
\def\Site#1{{\sf{#1}}}
\def\Object#1{\operatorname{Obj}\,(\text{\small\(#1\)})}
\def\Morphism#1{\operatorname{Mor}_{\text{\small\,\(#1\)\,}}}
\def\Covering#1{\operatorname{Cov}_{\text{\small\,\(#1\)\,}}}
\def\mathbold#1{\mbox{\boldmath $#1$}}
\def\manifold{\Category{Manifold}}
\def\diffeology{\Category{Diffeology}}

\def\gentop{\Category{NG}}

\def\sets{\Site{Set}}
\def\homeo{\approx}

\def\Paths{\operatorname{Paths}}
\def\Map{\operatorname{Map}}

\def\Int{\operatorname{Int}}
\def\Cl{\operatorname{Cl}}

\def\ast{\hbox{\footnotesize$*$}}

\def\colim{\operatorname{colim}\,}

\newcommand{\comp}{\smash{\lower-.1ex\hbox{\scriptsize$\circ$\,}}}
\def\smoothologytxt{diffeology}

\def\smoothologicaltxt{diffeological}

\makeatletter
\newenvironment{enumerate*}{\vskip.5ex
\begin{inparaenum}[i)\hspace{.2em}]
}{\end{inparaenum}\vskip.5ex\noindent}

\def\hitem{\hskip1em\item}
\makeatother
\begin{document}
\ifdefined\expand
\baselineskip21pt
\else
\ifdefined\narrow
\baselineskip15pt
\else
\baselineskip18pt
\fi
\fi
%
%
\title{Whitney Approximation for Smooth CW Complex}
%
%
\author{Norio IWASE}
\email{iwase@math.kyushu-u.ac.jp\vskip-1ex}
%
%
\address
{Faculty of Mathematics,
 Kyushu University,
 Fukuoka 819-0395, Japan\vskip-1ex}
%
%
%
%
%
\begin{abstract}
We show a Whitney Approximation Theorem for a continuous map from a manifold to a smooth CW complex.  This enables us to show that a topological CW complex is homotopy equivalent to a smooth CW complex in a category of topological spaces. 
It is also shown that, for any open covering of a smooth CW complex, there exists a partition of unity subordinate to the open covering.
\end{abstract}
%
%
\maketitle

\section{Introduction}\label{sect:Introduction}

One can collect ideas of smoothness and put them into a site a concrete category equipped with a coverage in the sense of Grothendieck.
K. T. Chen \cite{MR380859,MR377960,MR454968,MR842915} showed that it can be performed using the site $\convex{}$ the category of convex sets in Euclidean spaces, each of which has non-void interior, and smooth functions in the ordinary sense, where the covering family of a convex set is the set of coverings by open convex subsets.
A similar but a more sophisticated idea is employed by J. M. Souriau \cite{MR607688} (see also P. Iglesias-Zemmour \cite{MR3025051}) using an open set in place of an (open) convex set in the definition of a site $\convex{}$ which is denoted by $\domain{}$.

For a set $X$, let $\mathcal{M}_{X} : \domain{} \to \sets$ be the contravariant functor given by $\mathcal{M}_{X}(U)=\Map(U,X)$ for any open set $U \subset \real^{n}$, and $\mathcal{K}_{X}$ the subfunctor of $\mathcal{M}_{X}$ taking all constant maps.
A {\smoothologicaltxt} space is a set $X$ with a functor $\mathcal{D}_{X}$ satisfying the following:
\begin{enumerate}
\item[(D1)]\label{defn:smoothology-1}
\begin{enumerate*}
\item
\label{defn:smoothology-1-1}
$\mathcal{D}_{X}$ is a subfunctor of $\mathcal{M}_{X}$,
\hitem
\label{defn:smoothology-1-2}
$\mathcal{D}_{X}$ is a superfunctor of $\mathcal{K}_{X}$,
\end{enumerate*}
\item[(D2)]\label{defn:smoothology-2}
For any $U \!\in\! \Object{\domain{}}$ and any map $P : U \to X$, if there is $\{U_{\alpha}\}_{\alpha\in\Lambda} \in \Covering{\domain{}}(U)$ such that $P\vert_{U_{\alpha}} \in \mathcal{D}_{X}(U_{\alpha})$ for all $\alpha\in\Lambda$, then $P \in \mathcal{D}_{X}(U)$.
\end{enumerate}
Here, for a given site $\Site{Site}$, we denote by $\Object{\Site{Site}}$ the set of objects, by $\Morphism{\Site{Site}}(A,B)$ the set of morphisms from $A$ to $B$, and by $\Covering{\Site{Site}}(U)$ the set of covering families on $U$, in  $\Site{Site}$.
For a diffeological space $X$, $\mathcal{D}_{X}$ is often called a diffeology on $X$.

A map $f : X \to Y$ is said to be smooth if the composition with $f$ induces a map $f_{*} : \mathcal{D}_{X}(U) \to \mathcal{D}_{Y}(U)$ for all $U \in \Object{\domain{}}$, in other words, $f_{*}$ is a natural transformation from $\mathcal{D}_{X}$ to $\mathcal{D}_{Y}$.
An element of $\mathcal{D}_{X}(U)$ is called a plot of $X$ on $U$.
We denote by $\smoothology{}$, the category of diffeological spaces and smooth maps.

In this paper, a manifold is assumed always to be paracompact.
We denote by $\manifold$ the category of smooth manifolds and smooth maps which can be imbedded into $\smoothology{}$ as a full subcategory (see \cite{MR3025051}).
One of the advantage to expand our playground to $\smoothology{}$ rather than to restrict ourselves in $\manifold$ is that the category $\smoothology{}$ is cartesian-closed complete and cocomplete (see \cite{MR3025051}), so that we can play with tools imported from homotopy theory.
We also use the convenient category $\gentop$ of topological spaces introduced by Shimakawa-Yoshida-Haraguchi in \cite{MR3884529}.

Taking $D$-topology (see \cite{MR3025051}) gives a functor $D : \smoothology{} \to \gentop{}$.
But for a smooth manifold with or without boundary, or a smooth CW complex $X$, $D(X)$ will often be denoted by $X$ again in the following manner.
For two {\smoothologicaltxt} spaces $X$ and $Y$, a map $f : D(X) \to D(Y)$ in $\gentop$ is called a ``{continuous map}'' and denoted by ``{$f : X \to Y$}'' (in $\smoothology{}$).
For two continuous maps $f, g : X \to Y$, we say that a continuous map $H : D(X {\times} \real) \to D(Y)$ in $\gentop$, satisfying $H(x,t)=f(x)$ if $t \le \varepsilon$ and $H(x,t)=g(x)$ if $t \ge 1{-}\varepsilon$ for some $0\!<\!\varepsilon\!\ll\!1$, is called a ``{continuous homotopy}'' and denoted by ``{$H : X \times \real \to Y$}'' (in $\smoothology{}$).
In the latter case, $f$ and $g$ are said to be ``{continuously homotopic}''. 
A continuous map $f : X \to Y$ is called a ``{continuous homotopy equivalence}'' if there is a continuous map $h : Y \to X$ such that $f \comp h : Y \to Y$ and $h \comp f :X \to X$ are continuously homotopic to the identities $\id_{X}$ and $\id_{Y}$, resp.

\section{Main Results}\label{sect:MainResult}

A map $f$ from a manifold $N$ is said to be smooth on a closed subspace $A \subset N$ if $f$ is smooth on an open superset of $A$.
Let us recall the following theorem (cf. \cite{MR2954043}).

\begin{thm}[Whitney Approximation for Manifold]\label{thm:main-om}
Let $M$ and $X$ be smooth manifolds.
Then for a continuous map $f$ $:$ $M \to X$, there is a smooth map $g$ $:$ $M \to X$ and a homotopy from $f$ to $g$. If $f$ is smooth on a closed subset $A \subset M$, then the homotopy can be taken to be relative to $A$.
\end{thm}

The following statement is our main result.
\begin{thm}\label{thm:main}
Let $M$ be a smooth manifold and $X$ be a smooth CW complex.
Then for a continuous map $f$ $:$ $M \to X$, there exists a smooth map $g$ $:$ $M \to X$ and a continuous homotopy from $f$ to $g$.  If $f$ is smooth on a closed subset $A \subset M$ with an additional assumption that $M$ is compact or $X$ is of finite dimension, then the continuous homotopy can be taken to be relative to $A$.
\end{thm}

Theorem \ref{thm:main-om} is usually shown by using Whitney Approximation Theorem stated below together with the tubular neighborhood technique (cf. \cite{MR2954043}).
In this paper, we shall also use the theorem by adopting a skeleton-wise argument to show our main result:

\begin{thm}[Whitney Approximation Theorem]
\label{thm:main-oe}
Let $N$ be a smooth manifold and $n \!\ge\! 1$.
Then for continuous functions $f$ $:$ $N \to \real^{n}$ and $\kappa$ $:$ $N \to (0,\infty)$, there is a smooth function $g$ $:$ $N \to \real^{n}$ such that $d(f(x),g(x)) \!<\! \kappa(x)$ for $x \in N$.  If $f$ is smooth on a closed subset $A \subset N$, then $g$  can be chosen to be equal to $f$ on $A$.
\end{thm}

Now we restate \cite[Theorem A.1]{MR3991180} in this context as follows.

\begin{thm}\label{thm:main2}
A CW complex is homotopy equivalent in $\gentop$ to a smooth CW complex.\footnote{We realized that a similar result to Theorem \ref{thm:main2} for cubical or simplicial version of smooth CW complex is obtained by Haraguchi and Shimakawa \cite{arXiv:1912.05359,arXiv:1811.06175} or Hiroshi Kihara \cite{arXiv:1702.04070}.}
\end{thm}
\vskip-.5ex
\begin{Pf}
For any CW complex $Y$, we construct a smooth CW complex $X$ and a homotopy equivalence from $Y$ to $X$. 
To do this, we define a smooth CW pair $(X_{n+1},X_{n})$ from the CW pair $(Y_{n+1},Y_{n})$ and a continuous homotopy equivalence from $(Y_{n+1},Y_{n})$ to $(X_{n+1},X_{n})$ in $\gentop$, by induction on $n \!\ge\! 0$.
If $n=0$, we have nothing to do, and we proceed the case when $n{+}1 \!\ge\! 1$, assuming that we have constructed a smooth CW complex $X_{n}$ of dimension $n$ and a continuous homotopy equivalence $\phi_{n} : Y_{n} \to X_{n}$ in $\gentop$.
Let $Y_{n+1}=Y_{n} \cup_{k_{n}}\amalg_{\alpha}\,D^{n+1}_{\alpha}$, where $k_{n} : S_{n} \to Y_{n}$ is a continuous map in $\gentop$ and $S_{n} = \amalg_{\alpha}\,S^{n}_{\alpha}$ is a disjoint sum of $n$-spheres in $\manifold$.
Then for a continuous map $f_{n} = \phi_{n} \comp k_{n} : S_{n} \to X_{n}$ from a manifold to a smooth CW complex, Theorem \ref{thm:main} tells us that there is a smooth map $h_{n} : S_{n} \to X_{n}$ in $\smoothology{}$ which is continuously homotopic to $f_{n} : S_{n} \to X_{n}$.
Let $X_{n+1}=X_{n} \cup_{h_{n}} \amalg_{\alpha}\,D^{n+1}_{\alpha}$ a smooth CW complex of dimension $n{+}1$ in $\smoothology{}$.
Then, by standard arguments in homotopy theory, we obtain a homotopy equivalence $\phi_{n+1} : Y_{n+1} \to X_{n+1}$ extending $\phi_{n} : Y_{n} \to X_{n}$ in $\gentop$.
By putting $X=\underset{n}\colim\,X_{n}$, we obtain a smooth CW complex $X$ and a homotopy equivalence $\phi : Y \to X$ which is given by $\phi\vert_{Y_{n}}=\phi_{n}$ for all $n \!\ge\! 0$ in $\gentop$.
\end{Pf}

We expect that the following assertions are also true.

\begin{conj}\label{conj:main-m}
Let $X$ and $Y$ be smooth CW complexes.
Then, for a continuous map $f$ $:$ $Y \to X$, there exists a smooth map $g$ $:$ $Y \to X$ and a continuous homotopy from $f$ to $g$.  If $f$ is smooth on a closed subset $A \subset Y$, then the continuous homotopy can be taken to be relative to $A$.
\end{conj}
\begin{conj}\label{conj:main-m1}
Let $X$ and $Y$ be smooth CW complexes.
If they have the same topological homotopy type, then they have the same {\smoothologicaltxt} homotopy type.
\end{conj}

\section{$D$-Topology Of Smooth CW Complex}

Since $D$ is a left adjoint functor, it preserves colimits (see \cite{MR1712872}), and we obtain

\begin{prop}\label{prop:pushout}
If $Z$ is a pushout of $f_{1}$ and $f_{2}$ in $\smoothology{}$, then $D(Z)$ is the pushout of $D(f_{1})$ and $D(f_{2})$ in $\gentop$.
Further for an expanding sequence of {\smoothologicaltxt} spaces $\{X_{n}\}_{n \ge 0}$, we obtain $D(\underset{n}\colim\,X_{n}) = \underset{n}\colim\,D(X_{n})$ in $\gentop$.
\end{prop}

Now, let us recall the notion of smooth CW complex introduced in \cite{MR3991180}.
A smooth CW complex $X$ is a colimit of skeleta $\{X_{n}\}_{n\ge0}$ defined inductively as follows. 
\begin{itemize}
\item $X_{0}$ is a discrete set with discrete {\smoothologytxt}.
\vskip.5ex
\item For any $n\ge 0$, there is a smooth attaching map $h_{n}$ of all $n$-cells from a disjoint union of $n$-spheres $S_{n}=\amalg_{\alpha}\,S^{n}_{\alpha}$ to $X_{n}$ such that $X_{n+1} = X_{n} \cup_{h_{n}} \amalg_{\alpha}\, D^{n+1}_{\alpha}$, the pushout of $h_{n} : S_{n} \to X_{n}$ and the natural inclusion $S_{n} \hookrightarrow D_{n+1}=\amalg_{\alpha}\, D^{n+1}_{\alpha}$.
\end{itemize}
By definition, a smooth CW complex is naturally a topological CW complex with usual topology in $\gentop$, and a smooth manifold is also naturally a topological manifold with usual topology in $\gentop$, which we shall call their ``underlying topology''.
\begin{prop}\label{prop:smoothCWtopology}
Let $X$ be a smooth manifold, a disjoint sum of disks, or a smooth CW complex.
Then $D$-topology of $X$ is the same as the underlying topology of $X$.
\end{prop}
\vskip-.5ex
\begin{Pf}
For a smooth manifold or a disjoint sum of disks, the result follows easily from \cite[4.12]{MR3025051} (see Christensen-Sinnamon-Wu \cite[Example 3.2 (1) and Lemma 3.17]{MR3270173}).
For a smooth CW complex $X$ with $n$-skeleton $X_{n}$, we obtain, by \cite[2.11]{MR3025051}, that $D$-topology of $X_{0}$ is the underlying topology.
For $n \!\ge\! 0$, $D$-topology of $D_{n+1}$ is the underlying topology, since $D_{n+1}$ is a disjoint sum of disks.
Because $X_{n+1}$ is a pushout of $X_{n}$ and $D_{n+1}$, $D$-topology of $X_{n+1}$ is the pushout topology of $D(X_{n})$ and $D(D_{n+1})=\amalg_{\alpha}\,D(D^{n+1}_{\alpha})$ with the underlying topologies by Proposition \ref{prop:pushout}.
Again by Proposition \ref{prop:pushout}, we obtain that $D$-topology of $X$ is the weak topology filtered by $D(X_{n})$, which is the same as the underlying topology of the smooth CW complex $X$.
\end{Pf}

\section{Partition Of Unity}\label{sect:ParitionOfUnity}

Let $K = L \cup_{h} D_{n+1}$ be a $n$-cellular extension of a space $L$ by a smooth map $h : S_{n} \to L$ in $\diffeology{}$, where we assume that the $D$-topology $K$ is paracompact and Hausdorff.
Then $K$ is the pushout of a smooth map $h : S_{n}=\amalg_{\alpha}\, S^{n}_{\alpha} \to L$ and a smooth inclusion $\iota$ $:$ $S_{n} \hookrightarrow D_{n+1}=\amalg_{\alpha}\, D^{n+1}_{\alpha}$, and hence we have a subduction $\pi : L \amalg D_{n+1} \to K$.

Let $\mathcal{U}=\{U_{\lambda} \not=\emptyset \mid \lambda \in \Lambda\}$ be a given locally-finite $D$-open covering of $K$, $\Lambda^{D}_{\alpha}=\{\lambda \in \Lambda \mid D^{n+1}_{\alpha} \cap \pi^{-1}U_{\lambda} \not= \emptyset\}$, and $\Lambda_{L} = \{\lambda \in \Lambda \mid L \cap U_{\lambda} \not=\emptyset\}$.
Firstly, since $D^{n+1}_{\alpha}$ is a compact manifold, the locally-finite covering $\mathcal{U}^{D}_{\alpha}=\{D^{n+1}_{\alpha} \cap \pi^{-1}U_{\lambda} \mid \lambda \in \Lambda^{D}_{\alpha}\}$ is finite, and there is a smooth partition of unity $\{\rho^{D}_{\alpha,\lambda} \mid \lambda \in \Lambda^{D}_{\alpha}\}$ subordinate to $\mathcal{U}^{D}_{\alpha}$, satisfying $\supp \rho^{D}_{\alpha,\lambda} \subset \pi^{-1}U_{\lambda}$, $\lambda \in \Lambda^{D}_{\alpha}$. 
Secondly, we have a locally-finite $D$-open covering $\mathcal{U}_{L} = \{L \cap U_{\lambda} \mid \lambda \in \Lambda_{L}\}$ of $L$.
We remark that $\pi_{L} : L \to K$ is an inclusion.

\begin{prop}\label{prop:PartitionOfUnity}
If there is a smooth partition of unity $\{\rho_{\lambda} \mid \lambda \in \Lambda_{L}\}$ subordinate to $\mathcal{U}_{L}$ satisfying $\supp \rho_{\lambda} \subset L \cap U_{\lambda}$, $\lambda \in \Lambda_{L}$, then there exists a smooth partition of unity $\{\hat\rho_{\lambda} \mid \lambda \in \Lambda\}$ subordinate to $\mathcal{U}$ satisfying $\supp \hat\rho_{\lambda} \subset U_{\lambda}$, $\lambda \in \Lambda$ and $\hat\rho_{\lambda}|_{L}=\rho_{\lambda}$, $\lambda \in \Lambda_{L}$.
\end{prop}
\begin{proof}
Let $\Lambda^{S}_{\alpha} = \{\lambda \in \Lambda \mid S^{n}_{\alpha} \cap \pi^{-1}U_{\lambda} \not=\emptyset\}$ and then we obtain a finite open covering $\mathcal{U}^{S}_{\alpha}=\{\pi^{-1}U_{\lambda} \mid \lambda \in \Lambda^{S}_{\alpha}\}$ of $S^{n}_{\alpha}$ in $D^{n+1}_{\alpha}$ such that $h_{\alpha}^{-1} \supp \rho_{\lambda} \subset \pi^{-1}U_{\lambda}$, $\lambda \in \Lambda^{S}_{\alpha}$, where $h_{\alpha}=h|_{S^{n}_{\alpha}}$.
Since $S^{n}_{\alpha}$ is compact and $\Lambda^{S}_{\alpha}$ is finite, there exists a positive number $\epsilon\!=\!\epsilon_{\alpha} \ll 1$ such that the $\epsilon$-neighbourhood of $h_{\alpha}^{-1} \supp \rho_{\lambda}$ in $D^{n+1}_{\alpha}$ is included in $\pi^{-1}U_{\lambda}$, $\lambda \in \Lambda^{S}_{\alpha}$.

Let $S^{n}_{\alpha,\epsilon} := \{\mathbold{v} \!\in\! D^{n+1}_{\alpha} \mid \Vert{\mathbold{v}}\Vert > 1{-}{\epsilon} \}$ and let smooth functions $\{\rho^{S}_{\alpha,\lambda} \mid \lambda \in \Lambda^{S}_{\alpha}\}$ on $S^{n}_{\alpha,\epsilon}$ be defined, for $\mathbold{v} \in S^{n}_{\alpha,\epsilon}$, by 
$$\textstyle
\rho^{S}_{\alpha,\lambda}(\mathbold{v}) = \begin{cases}
\rho_{\lambda} \comp h_{\alpha} \comp p_{\alpha}(\mathbold{v}),& \text{if $\mathbold{v} \in p_{\alpha}^{-1}h_{\alpha}^{-1}U_{\lambda}$}, 
\\
0,&\text{otherwise,}
\end{cases}
$$
where $p_{\alpha}$ is defined by $p_{\alpha}(\mathbold{v}) = \textstyle\frac{\mathbold{v}}{\Vert{\mathbold{v}}\Vert}$ for $\mathbold{v} \in S^{n}_{\alpha,\epsilon}$.
Here, $\frac{\mathbold{v}}{\Vert{\mathbold{v}}\Vert} \in h_{\alpha}^{-1}\supp \rho_{\lambda}$ implies that $\mathbold{v} \in \pi^{-1}U_{\lambda}$, since $\Vert{\mathbold{v}{-}\frac{\mathbold{v}}{\Vert{\mathbold{v}}\Vert}}\Vert = 1{-}\Vert{\mathbold{v}}\Vert < \epsilon$, $\lambda \in \Lambda^{S}_{\alpha}$.
Thus $\supp \rho^{S}_{\alpha,\lambda} \subset p^{-1}_{\alpha} h^{-1}_{\alpha} \supp \rho_{\lambda} \subset S^{n}_{\alpha,\epsilon} \cap \pi^{-1}U_{\lambda}$, $\lambda \in \Lambda^{S}_{\alpha}$.
By definition, we also obtain $\sum_{\lambda \in \Lambda^{S}_{\alpha}} \rho^{S}_{\alpha,\lambda}(\mathbold{v}) = \sum_{\lambda \in \Lambda^{S}_{\alpha}}\rho_{\lambda}(h_{\alpha} \comp p_{\alpha}(\mathbold{v})) = 1$, and hence $\{\rho^{S}_{\alpha,\lambda}\}$ is a partition of unity subordinate to $\mathcal{U}^{S}_{\alpha,\epsilon}=\{S^{n}_{\alpha,\epsilon} \cap \pi^{-1}U_{\lambda} \mid \lambda \in \Lambda^{S}_{\alpha}\}$. 

On the other hand, $\{S^{n}_{\alpha,\epsilon}, \Int{D^{n+1}_{\alpha}}\}$ is an open covering of $D^{n+1}_{\alpha}$, and we have a partition of unity $\{\chi^{S}_{\alpha},\chi^{D}_{\alpha}\}$ on $D^{n+1}_{\alpha}$, such that $\supp \chi^{S}_{\alpha} \subset S^{n}_{\alpha,\epsilon}$ and $\supp \chi^{D}_{\alpha} \subset \Int{D^{n+1}_{\alpha}}$.
We define a set of smooth functions $\{\hat\rho^{D}_{\alpha,\lambda} \mid \lambda \in \Lambda\}$ on $D^{n+1}_{\alpha}$ by the following formula:
$$
\hat\rho^{D}_{\alpha,\lambda} = \begin{cases}
\chi^{D}_{\alpha} \cdot \rho^{D}_{\alpha,\lambda} + \chi^{S}_{\alpha} \cdot \rho^{S}_{\alpha,\lambda}, & \text{if $\lambda \in \Lambda^{S}_{\alpha}$,}
\\
\chi^{D}_{\alpha} \cdot \rho^{D}_{\alpha,\lambda}, & \text{if $\lambda \in \Lambda^{D}_{\alpha} \smallsetminus \Lambda^{S}_{\alpha}$,}
\\
0,&\text{otherwise.}
\end{cases}
$$
Thus on $D^{n+1}_{\alpha}$, we have $\sum_{\lambda \in \Lambda^{D}_{\alpha}} \hat\rho^{D}_{\alpha,\lambda} = \chi^{D}_{\alpha} \cdot \sum_{\lambda \in \Lambda^{D}_{\alpha}}\rho^{D}_{\alpha,\lambda} + \chi^{D}_{\alpha} \cdot \sum_{\lambda \in \Lambda^{S}_{\alpha}}\rho^{S}_{\alpha,\lambda} = \chi^{D}_{\alpha} + \chi^{S}_{\alpha} = 1$, and hence $\{\hat\rho^{D}_{\alpha,\lambda} \mid \lambda \in \Lambda\}$ gives a smooth partition of unity subordinate to $\mathcal{U}^{D}_{\alpha}$ satisfying $\supp \hat\rho^{D}_{\alpha,\lambda} \subset \supp \rho^{D}_{\alpha,\lambda} \cup \supp \rho^{D}_{\alpha,\lambda} \subset \pi^{-1}U_{\lambda}$, $\lambda \in \Lambda^{D}_{\alpha}$.
Further by definition, we have
$$
\hat\rho^{D}_{\alpha,\lambda}\vert_{S^{n}_{\alpha}} = \begin{cases}
\chi^{S}_{\alpha} \cdot \rho^{S}_{\alpha,\lambda}\vert_{S^{n}_{\alpha}} = \rho^{S}_{\alpha,\lambda}\vert_{S^{n}_{\alpha}} = \rho_{\lambda} \comp h_{\alpha}, & \text{if $\lambda \in \Lambda^{S}_{\alpha}$,}
\\
0 = \rho_{\lambda} \comp h_{\alpha},&\text{otherwise.}
\end{cases}
$$
In each case, we have $\hat\rho^{D}_{\alpha,\lambda}\vert_{S^{n}_{\alpha}} =  \rho_{\lambda} \comp h_{\alpha}$, and hence we obtain a smooth function $\hat\rho_{\lambda}$ on $K$ given by $\hat\rho_{\lambda}|_{L}=\begin{cases}\rho_{\lambda},&\lambda \in \Lambda_{L}\\0,&\text{otherwise}\end{cases}$ and $\hat\rho_{\lambda} \comp \pi|_{D^{n+1}_{\alpha}} = \begin{cases}\hat\rho^{D}_{\alpha,\lambda},& \lambda \in \Lambda^{D}_{\alpha}\\0,&\text{otherwise}\end{cases}$.
They clearly give a smooth partition of unity subordinate to $\mathcal{U}$ satisfying $\supp \hat\rho_{\lambda} \subset U_{\lambda}$, $\lambda \in \Lambda$.
\end{proof}

By inductively using the above proposition, we obtain the following result.

\begin{thm}\label{thm:PartitionOfUnity}
For any $D$-open covering of a smooth CW complex, there exists a partition of unity subordinate to the covering.
\end{thm}
\begin{proof}
Let $\mathcal{V}$ be a given $D$-open covering of $X$.
Then, since $D(X)$ is paracompact and Hausdorff, there is a locally-finite  refinement $\mathcal{U}=\{U_{\lambda} \mid \lambda \in \Lambda\}$ of $\mathcal{V}$ such that, for any $\lambda$, there is  $V \in \mathcal{V}$ satisfying $\Cl{U_{\lambda}} \subset V$.
So it is sufficient to show the existence of a smooth partition of unity $\{\rho_{\lambda} \mid \lambda \in \Lambda\}$ subordinate to $\mathcal{U}$ satisfying $\supp \rho_{\lambda} \subset \Cl{U_{\lambda}}$, $\lambda \in \Lambda$.

Let $\Lambda_{n} = \{\lambda \in \Lambda \mid U_{\lambda} \cap X_{n} \not= \emptyset\}$ and $\mathcal{U}_{n}=\{U_{\lambda} \cap X_{n} \mid \lambda \in \Lambda_{n}\}$.
Since $(\Lambda_{n+1})_{X_{n}}=\Lambda_{n}$ and $(\mathcal{U}_{n+1})_{X_{n}}=\mathcal{U}_{n}$, we obtain a series of smooth partitions of unity $\{\rho^{n}_{\lambda} \mid \lambda \in \Lambda_{n}\}$ on $X_{n}$ the $n$-skeleton of $X$, for all $n \ge 0$, such that $\rho^{n+1}_{\lambda}|X_{n}=\rho^{n}_{\lambda}$, $\lambda \in \Lambda_{n}$, by inductively using Proposition \ref{prop:PartitionOfUnity}.
Then, for any $\lambda \in \Lambda$, $\rho^{n}_{\lambda}$ exists for sufficiently large $n \ge 0$, and $\{\rho^{n}_{\lambda}\}$ gives a smooth function $\rho_{\lambda}$ on $X$, satisfying $\supp \rho_{\lambda} = \Cl{(\bigcup_{n} \supp \rho^{n}_{\lambda})} \subset \Cl{(\bigcup_{n} X_{n} \cap U_{\lambda})} = \Cl{U_{\lambda}}$.
We also have, by definition, $\sum_{\lambda}\rho_{\lambda}=1$ on $X_{n}$ for all $n \ge 0$, $X = \underset{n}\colim X_{n}$, and hence $\{\rho^{n}_{\lambda}\}$ is a partition of unity on $X$ subordinate to $\mathcal{V}$.
\end{proof}

We say a diffeological space has enough many smooth functions, if its $D$-topology has an open base of the form $\pi^{-1}(]0,1[)$. By Theorem \ref{thm:PartitionOfUnity}, we clearly obtain the following.

\begin{cor}
A smooth CW complex has enough many smooth functions.
\end{cor}

\section{Proof of Theorem \ref{thm:main}}\label{sect:ProofMainTheorem}

For a {\smoothologicaltxt} space $X$, we denote by $\Paths(X)$ the mapping space of all smooth maps from $\real$ to $X$, following  \cite{MR3025051}.
Since $\smoothology{}$ is cartesian-closed, taking adjoint gives a natural bijection as follows.
\vspace{-.5ex}
$$
\Morphism{\smoothology{}}(X \times \real,Y) \ni f \overset{g=\ad{f}}\longleftrightarrow g \in \Morphism{\smoothology{}}(X,\Paths(Y)),
$$
where the map $\ad{f} : X \to \Paths(Y)$ is defined by $\ad{f}(x)(t)=f(x,t)$, \,$(x,t) \in X \times \real$ for a map $f : X \times \real \to Y$ in $\smoothology{}$.

When $A$ is an empty set, since a manifold $M$ is of dimension $d$, then a continuous map from $M$ to a smooth CW complex $X$ can be continuously compressed into $X_{d}$ in $\gentop$.
When $A$ is a non-empty set, we must assume either that $M$ is compact or that $X$ is of finite dimension.
In case when $M$ is compact, $\Im{f}$ is compact in $X$, and hence $\Im{f} \subset X_{d}$ for some $d \!\ge\! 0$.
In case when $X$ is of finite dimension, we have $\Im{f} \subset X=X_{d}$ for $d \!=\! \dim{X}$.
In either case, we may assume that $\Im{f} \subset X_{d}$ for some $d \!\ge\! 0$.
To proceed further, we use induction on $d \!\ge\! 0$.
Since it is clear when $d=0$, we assume that we have done in the case when $d \!\le\! n$, from now on.

Let $\tau =\fracinline1/6$, $\eta = \fracinline{1}/{12}$, $0\!<\!\varepsilon\!\le\!\fracinline1/6$, and $\lambda : \real \to \real$ be a smooth function satisfying 
\begin{enumerate}
\item $\lambda(t) = 0$ \ if \ $t \le \varepsilon$, 
\item $\lambda$ is an increasing monotone function on $(\varepsilon,1\!-\!\varepsilon)$,
\item $\lambda(1{-}t)=1\!-\!\lambda(t)$ \ for all $t \in \real$.
\end{enumerate}
Firstly, let $\{e^{n+1}_{\alpha} ; \alpha {\in} \Lambda\}$ be the set of $n{+}1$-cells of $X_{n+1}$, and $X_{n+1}=X_{n} \cup_{h_{n}} D_{n+1}$, where $(D_{n+1},S_{n})$ $=$ $(\amalg_{\alpha}\,D^{n+1}_{\alpha},\amalg_{\alpha\in\Lambda}\,S^{n}_{\alpha})$ and $h_{n} : S_{n} \to X_{n}$ is the smooth attaching map of all $n{+}1$ cells. 
Then there is a smooth characteristic map $\chi_{n+1} : (D_{n+1},S_{n}) \to (X_{n+1},X_{n})$ in $\smoothology{}$.
We define subsets of $X_{n+1}$ in $\gentop$ as $O:=\amalg_{\alpha}\,O_{\alpha}$, $U^{(t)}:=\amalg_{\alpha}\,U^{(t)}_{\alpha}$, $U:=U^{(0)}=\amalg_{\alpha}\,U_{\alpha}$, $V^{(t)}:=X_{n} \cup_{h_{n}} \amalg_{\alpha}\,V^{(t)}_{\alpha}$ and $V:=V^{(0)}=X_{n} \cup_{h_{n}} \amalg_{\alpha}\,V_{\alpha}$, for $\vert{t}\vert<\fracinline{1}/{2}$, where 
the subsets $O_{\alpha}$, \,$U_{\alpha}$, \,$U^{(t)}_{\alpha}$, \,$V_{\alpha}$ and $V^{(t)}_{\alpha}$ of $X_{n+1}$ are defined as follows: 
$$\begin{array}{l}
O_{\alpha}:=D^{n+1}_{\alpha} \smallsetminus S^{n}_{\alpha} \homeo e^{n+1}_{\alpha},
\qquad%
U^{(t)}_{\alpha} := \{\mathbold{x} \in D^{n+1}_{\alpha} \mid \Vert{\mathbold{x}}\Vert{<}\fracinline{1}/{2}{+}t\},\qquad U_{\alpha}:=U^{(0)}_{\alpha},
\\[1ex]
V^{(t)}_{\alpha} := \{\mathbold{x} \in D^{n+1}_{\alpha} \mid \Vert{\mathbold{x}}\Vert{>}\fracinline{1}/{2}{+}t\},\qquad V_{\alpha}:=V^{(0)}_{\alpha}.
\end{array}$$\vskip.5ex
Then Proposition \ref{prop:smoothCWtopology} tells us the following.
\begin{prop}
$O$, $U$, $U^{(t)}$, $V$, $V^{(t)}$ for $\vert{t}\vert\!<\!\fracinline{1}/{2}$, are $D$-open subsets of $X_{n+1}$.
\end{prop}
Let $\mathcal{U}=\{f^{-1}(U^{(2\tau)}),f^{-1}(V^{(\tau)})\}$ be an open covering of $M$.  Since $M$ is a manifold, there is a smooth partition of unity $\rho_{U},\,\rho_{V} : M \to [0,1]$ subordinate to $\mathcal{U}$, i.e., $\rho_{U}$ and $\rho_{V}$ are smooth functions satisfying $\supp\rho_{U} \subset f^{-1}(U^{(2\tau)})$, $\supp{\rho_{V}} \subset f^{-1}(V^{(\tau)})$ and $\rho_{U}+\rho_{V}=1$ on $M$.
By the hypothesis of Theorem \ref{thm:main}, we have an open superset $B \subset M$ of $A$, on which $f$ is smooth in $\smoothology{}$.
For any $\alpha\!\in\!\Lambda$, by Theorem \ref{thm:main-oe} for a continuous map $f\vert_{f^{-1}(U^{(2\tau)}_{\alpha})} : f^{-1}(U^{(2\tau)}_{\alpha}) \to U^{(2\tau)}_{\alpha} \!\subset\! \real^{n+1}$ and a constant function $\kappa : f^{-1}(U^{(2\tau)}_{\alpha}) \ni x \to \tau$, there is a smooth function $g_{\alpha} : f^{-1}(U^{(2\tau)}_{\alpha}) \to \real^{n+1}$, such that $d(g_{\alpha}(x),f(x))<\kappa(x)=\tau$ for all $x \in f^{-1}(U^{(2\tau)}_{\alpha})$, and that $g_{\alpha}=f$ on the closed subset $A_{\alpha}=f^{-1}(U^{(2\tau)}_{\alpha}) \cap A$ of $f^{-1}(U^{(2\tau)}_{\alpha})$.
We define a map $g_{1} : M \to X_{n+1}$ by 
$$
g_{1}(x) = 
\begin{cases}\,
\rho_{U}(x){\cdot}g_{\alpha}(x) + \rho_{V}(x){\cdot}f(x),&x \in f^{-1}(U^{(2\tau)}_{\alpha}) 
\supset O_{\alpha}\cap\supp{\rho_{U}}, \ \ \alpha \in \Lambda,
\\[.5ex]\,
f(x),& x \in M \smallsetminus \supp{\rho_{U}} \supset M \smallsetminus f^{-1}(U^{(2\tau)}).
\end{cases}
$$
\begin{prop}\label{prop:g-1-openset}
$f^{-1}(U^{(2\tau)}) \subset g_{1}^{-1}(O)=f^{-1}(O)$ and $f^{-1}(V^{(\tau)}) \subset g_{1}^{-1}(V)$.
\end{prop}
\vskip-.5ex
\begin{Pf}
Firstly, if $x \in f^{-1}(U^{(2\tau)}_{\alpha})$, then $\Vert{f(x)}\Vert < 5\tau$ and hence $\Vert{g_{1}(x)}\Vert \le \Vert{f(x)}\Vert + d(g_{1}(x),f(x)) < 5\tau \!+\! \tau = 6\tau = 1$ which implies $x \in g_{1}^{-1}(O_{\alpha})$.
Thus we obtain $f^{-1}(U^{(2\tau)}) \subset g_{1}^{-1}(O)$.
Secondly, if $x \in f^{-1}(O_{\alpha} \smallsetminus U^{(2\tau)}_{\alpha})$, then $g_{1}(x) = f(x) \in O_{\alpha} \smallsetminus U^{(2\tau)}_{\alpha}$, and hence $f^{-1}(O) = f^{-1}(O \smallsetminus U^{(2\tau)}) \cup f^{-1}(U^{(2\tau)}) \subset g_{1}^{-1}(O \smallsetminus U^{(2\tau)}) \cup g_{1}^{-1}(O) = g_{1}^{-1}(O)$.
Conversely if $x \in M \smallsetminus f^{-1}(O)$, then $g_{1}(x)=f(x) \in X_{n}$, and hence $x \in M \smallsetminus g_{1}^{-1}(O)$.
Thus $f^{-1}(O) = g_{1}^{-1}(O)$.
Thirdly, if $x \in f^{-1}(V^{(\tau)}_{\alpha} \cap U^{(2\tau)}_{\alpha})$, then $4\tau <\Vert{f(x)}\Vert < 5\tau$ and $\fracinline{1}/{2}=3\tau < \Vert{g_{1}(x)}\Vert <6\tau=1$ which implies $x \in g_{1}^{-1}(V_{\alpha} \cap O_{\alpha})$.
On the other hand, if $x \in f^{-1}(V^{(\tau)} \smallsetminus U^{(2\tau)})$, $g_{1}(x)=f(x) \in V^{(\tau)} \smallsetminus U^{(2\tau)}$, and hence $f^{-1}(V^{(\tau)}) = f^{-1}(V^{(\tau)} \cap U^{(2\tau)}) \cup f^{-1}(V^{(\tau)} \smallsetminus U^{(2\tau)}) \subset g_{1}^{-1}(V \cap O) \cup g_{1}^{-1}(V \smallsetminus U^{(2\tau)}) = g_{1}^{-1}(V)$.
\end{Pf}
\begin{cor}\label{cor:g-1-openset}
\begin{enumerate}
\item $f^{-1}(X_{n})=M \smallsetminus f^{-1}(O)=M \smallsetminus g_{1}^{-1}(O)=g_{1}^{-1}(X_{n})$.
\item $M \smallsetminus \supp{\rho_{U}} \supset  M \smallsetminus f^{-1}(U^{(2\tau)}) = f^{-1}(\Cl{(V^{(2\tau)})}) \supset f^{-1}(V^{(2\tau)}) \supset f^{-1}(X_{n}) = g_{1}^{-1}(X_{n})$.
\item Since $f^{-1}(V^{(\tau)}) \supset \supp{\rho_{V}} \supset f^{-1}(X_{n})$, Proposition 4.2 implies $M \smallsetminus \supp{\rho_{V}} \supset f^{-1}(O) \smallsetminus f^{-1}(V^{(\tau)}) \supset g_{1}^{-1}(O) \smallsetminus g_{1}^{-1}(V) = g_{1}^{-1}(O \smallsetminus V) = g_{1}^{-1}(\Cl{U})$.
\end{enumerate}
\end{cor}
\begin{prop}\label{prop:g-1-properties}
\begin{enumerate}
\item\label{prop:g-1-properties-1} $f$ is homotopic to $g_{1}$ relative to $(M \smallsetminus \supp{\rho_{U}}) \cup A$.
\item\label{prop:g-1-properties-2} $g_{1}$ is smooth on an open superset $B$ of $A$ in $\smoothology{}$.
\item\label{prop:g-1-properties-3} $M \smallsetminus \supp{\rho_{V}} \subset g_{1}^{-1}(O)$ and $g_{1}\vert_{M \smallsetminus \supp{\rho_{V}}}$ is smooth in the ordinary sense.
\end{enumerate}
\end{prop}
\vskip-.5ex
\begin{Pf}
(\ref{prop:g-1-properties-1}):
A map $H_{1} : M \times \real \to X_{n+1}$ is given, for $t \in \real$, by using Corollary \ref{cor:g-1-openset} as:
\begin{align*}&
H_{1}(x,t) = \begin{cases}\,
(1{-}\lambda(t)){\cdot}f(x) + \lambda(t){\cdot}g_{1}(x),& x \in f^{-1}(U^{(2\tau)}) \supset \supp{\rho_{U}},
\\[.5ex]\,
f(x),& x \in M \smallsetminus \supp{\rho_{U}} \supset f^{-1}(\Cl{(V^{(2\tau)})}).
\end{cases}
\end{align*}
If $x \!\in\! (M \smallsetminus \supp{\rho_{U}}) \cap f^{-1}(U^{(2\tau)})$, then by definition, we obtain that $g_{1}(x) = f(x)$ and $(1{-}\lambda(t)){\cdot}f(x) + \lambda(t){\cdot}g_{1}(x)=f(x)$ for all $t \in \real$.
Thus $H_{1}$ is a well-defined continuous homotopy from $f$ to $g_{1}$.
By the hypothesis on $A$, we also have $f=g_{1}$ on $A$ and hence the homotopy is defined to be relative to $(M \smallsetminus \supp{\rho_{U}}) \cup A$.
\vskip.5ex\par(\ref{prop:g-1-properties-2}):
We know that $f$ is smooth on $B$ in $\smoothology{}$ and that $g_{\alpha}$ is smooth on $f^{-1}(U^{(2\tau)}_{\alpha})$ by definition.
Hence both of a map $\rho_{V}{\cdot}f+\rho_{U}{\cdot}g_{\alpha}$ on $B \cap f^{-1}(U^{(2\tau)}_{\alpha})$ for any $\alpha$ and a map $f$ on $B \cap (M \smallsetminus \supp{\rho_{U}})$ are smooth in $\smoothology{}$.
Since $\{B \cap f^{-1}(U^{(2\tau)}_{\alpha})\}_{\alpha\in\Lambda} \amalg \{B \cap (M \smallsetminus \supp{\rho_{U}})\}$ is an open covering of $B$, $g_{1}$ is smooth on $B$ in $\smoothology{}$.
\vskip.5ex\par(\ref{prop:g-1-properties-3}):
By definition, $g_{1}$ agrees with $g_{\alpha}$ on $f^{-1}(O_{\alpha}) \smallsetminus \supp{\rho_{V}}$ for all $\alpha \in \Lambda$.
Since $\supp{\rho_{V}} \supset f^{-1}(X_{n})$, we have $M \smallsetminus \supp{\rho_{V}} = f^{-1}(O) \smallsetminus \supp{\rho_{V}} = g_{1}^{-1}(O) \smallsetminus \supp{\rho_{V}} = \amalg_{\alpha}\,(f^{-1}(O_{\alpha}) \smallsetminus \supp{\rho_{V}})$ on which $g_{1}=\amalg_{\alpha}\,g_{\alpha}$ is smooth in the ordinary sense.
\end{Pf}

Next, we choose a smooth partition of unity $\{\rho'_{U}, \rho'_{V}\}$ subordinate to $\mathcal{U}'=\{U, V^{(-\tau)} \cap D_{n+1}\}$ an open covering of $D_{n+1}$, in other words, $\rho'_{U}$ and $\rho'_{V}$ are smooth functions satisfying $\supp{\rho'_{U}} \subset U$, $\supp{\rho'_{V}} \subset V^{(-\tau)}$ and $\rho'_{U}+\rho'_{V}=1$ on $D_{n+1}$.
Then a continuous map $K_{n+1} : D_{n+1} \times \real \to D_{n+1}$ is given, for $t \in \real$, by using Corollary \ref{cor:g-1-openset} as:
$$
K_{n+1}(y,t) = \begin{cases}\,
(1{+}\lambda(t)){\cdot}y \in O,& y \in D_{n+1} \smallsetminus \supp{\rho'_{V}}, 
\\[1ex]\,
((1{+}\lambda(t))\rho'_{U}(y)+\frac{(1{-}\lambda(t))\Vert{y}\Vert{+}\lambda(t)}{\Vert{y}\Vert}\rho'_{V}(y)){\cdot}y \in D_{n+1},& y \in U \cap V^{(-\tau)},
\\[1ex]\,
\frac{(1{-}\lambda(t))\Vert{y}\Vert{+}\lambda(t)}{\Vert{y}\Vert}{\cdot}y \in D_{n+1},& y \in D_{n+1} \smallsetminus \supp{\rho'_{U}}.
\end{cases}
$$
If $y \in U \supset D_{n+1} \smallsetminus \supp{\rho'_{V}}$, then $\Vert{y}\Vert < \fracinline1/2$, and $(1{+}\lambda(t)){\cdot}y \in O$. If $y \in V_{(-\tau)} \supset D_{n+1} \smallsetminus \supp{\rho'_{U}}$, then $\fracinline1/3 < \Vert{y}\Vert\le1$, and by $\Vert{y}\Vert \le {(1{-}\lambda(t))\Vert{y}\Vert{+}\lambda(t)}\le1$, $\frac{(1{-}\lambda(t))\Vert{y}\Vert{+}\lambda(t)}{\Vert{y}\Vert}{\cdot}y\in D_{n+1}$.
Hence, if $y \in U \cap V^{(-\tau)}$, $((1{+}\lambda(t))\rho'_{U}(y)+\frac{(1{-}\lambda(t))\Vert{y}\Vert{+}\lambda(t)}{\Vert{y}\Vert}\rho'_{V}(y)){\cdot}y \in D_{n+1}$.

Further, if $y \in (U \cap V^{(-\tau)}) \cap (D_{n+1} \smallsetminus \supp{\rho'_{V}})$, then $\rho'_{V}(y)=0$, $\rho'_{U}(y)=1{-}\rho'_{V}(y)=1$ and $(1{+}\lambda(t))\rho'_{U}(y)+\frac{(1{-}\lambda(t))\Vert{y}\Vert{+}\lambda(t)}{\Vert{y}\Vert}\rho'_{V}(y)$ $=$ $1{+}\lambda(t)$. Also if $y \in (U \cap V^{(-\tau)}) \cap (D_{n+1} \smallsetminus \supp{\rho'_{U}})$, then $\rho'_{U}(y)=0$, $\rho'_{V}(y)$ $=$ $1{-}\rho'_{U}(y)$ $=$ $1$ and $(1{+}\lambda(t))\rho'_{U}(y)+\frac{(1{-}\lambda(t))\Vert{y}\Vert{+}\lambda(t)}{\Vert{y}\Vert}\rho'_{V}(y)$ $=$ $\frac{(1{-}\lambda(t))\Vert{y}\Vert{+}\lambda(t)}{\Vert{y}\Vert}$.

By definition, we have that $K_{n+1}(y,0)=y$ for all $y \in D_{n+1}$, and $K_{n+1}(y,1)=\frac{1}{\Vert{y}\Vert}{\cdot}y$, \ if \ $y \in D_{n+1} \smallsetminus \supp{\rho'_{U}} \supset D_{n+1} \cap \Cl{V} \supset S_{n}$.
In particular, $K_{n+1}(y,1)=y$ for any $y \in S_{n}$.

We define $k_{n+1} : (D_{n+1},S_{n}) \to (D_{n+1},S_{n})$ by $k_{n+1}(y)=K_{n+1}(y,1)$ for $y \in D_{n+1}$.
Then $K_{n+1}$ is a well-defined homotopy from $\id_{D_{n+1}}$ to $k_{n+1}$ with the following property.
\begin{prop}\label{prop:K-properties}
$K_{n+1}$ is a smooth deformation of $D_{n+1}$ relative to $S_{n}$.
\end{prop}
\vskip-.5ex
\begin{Pf}
Since $\rho'_{U}$ and $\rho'_{V}$ are smooth, so are $(1{+}\lambda(t))\rho'_{U}(y)$ and $\frac{(1{-}\lambda(t))\Vert{y}\Vert{+}\lambda(t)}{\Vert{y}\Vert}\rho'_{V}(y)$, and hence $K_{n+1}$ is smooth.
Other parts of the statement are clear by definition.
\end{Pf}
The smooth map $k_{n+1} : (D_{n+1},S_{n}) \to (D_{n+1},S_{n})$ is given as follows.
$$
k_{n+1}(y) = K_{n+1}(y,1) = \begin{cases}\,
2{\cdot}y \in O,& y \in D_{n+1} \smallsetminus \supp{\rho'_{V}}, 
\\[.5ex]\,
2\rho'_{U}(y)+\frac{1}{\Vert{y}\Vert}\rho'_{V}(y)){\cdot}y \in D_{n+1},& y \in U \cap V^{(-\tau)},
\\[.5ex]\,
\frac{1}{\Vert{y}\Vert}{\cdot}y \in S_{n},& y \in D_{n+1} \smallsetminus \supp{\rho'_{U}}.
\end{cases}
$$

We know that $X_{n+1}$ is defined with smooth maps $\chi_{n+1} : D_{n+1} \to X_{n+1}$ and $\incl_{X_{n}} : X_{n} \hookrightarrow X_{n+1}$ as a pushout of smooth maps $\incl_{S_{n}} : S_{n} \hookrightarrow D_{n+1}$ and $h_{n} : S_{n} \rightarrow X_{n}$ in $\smoothology{}$, where they satisfy the equation $\chi_{n+1}\vert_{S^{n}}=\incl_{X_{n}} \comp h_{n}$.

Since $K_{n+1}$ is smooth in $\smoothology{}$ by Proposition \ref{prop:K-properties}, so is its adjoint $\ad{K_{n+1}} : D_{n+1} \to \Paths(D_{n+1})$, since $\smoothology{}$ is cartesian-closed.
We extend $\ad{K_{n+1}}$ and $k_{n+1}$ to obtain smooth maps $\ad{\widetilde{K}_{n+1}} : X_{n+1} \to \Paths(X_{n+1})$ and $\widetilde{k}_{n+1} : X_{n+1} \to X_{n+1}$ which are determined by the following data.
\begin{align*}&
\ad{\widetilde{K}_{n+1}} \comp \chi_{n+1}=(\chi_{n+1})_{\ast} \comp \ad{K_{n+1}} : D_{n+1} \to \Paths(D_{n+1}) \to \Paths(X_{n+1}), 
\\[0ex]&
\ad{\widetilde{K}_{n+1}}\vert_{X_{n}}={\incl_{X_{n}}}_{\ast} \comp \!\const{} : X_{n} \to \Paths(X_{n}) \hookrightarrow \Paths(X_{n+1}),
\\[.5ex]&
\widetilde{k}_{n+1} \comp \chi_{n+1}=\chi_{n+1} \comp k_{n+1} : D_{n+1} \to D_{n+1} \to X_{n+1}, \quad\text{and} 
\\[0ex]&
\widetilde{k}_{n+1}\vert_{X_{n}}=\incl_{X_{n}} : X_{n} \hookrightarrow X_{n+1},
\end{align*}
where $\const{} : Y \to \Paths(Y)$ sends $y \in Y$ to $\const{y} \in \Paths(Y)$ the constant path at $y$. 

Since ${\chi_{n+1}}_{\ast} \comp \ad{K_{n+1}}\vert_{S_{n}} = {\chi_{n+1}}_{\ast} \comp \!\const{}\vert_{S_{n}} = {(\chi_{n+1}\vert_{S_{n}})}_{\ast} \comp \!\const{} = {(\incl_{X_{n}} \comp h_{n})}_{\ast} \comp \!\const{} = {\incl_{X_{n}}}_{\ast} \comp \!\const{} \comp h_{n}$ and $\chi_{n+1} \comp k_{n+1}\vert_{S_{n}} = \chi_{n+1}\vert_{S_{n}} = \incl_{X_{n}} \comp h_{n}$, maps $\ad{\widetilde{K}_{n+1}}$ and $\widetilde{k}_{n+1}$ are well-defined and smooth in $\smoothology{}$.
Thus we have a smooth map $\widetilde{K}_{n+1} : X_{n+1} \times \real \to X_{n+1}$ the adjoint of $\ad{\widetilde{K}_{n+1}}$ in $\smoothology{}$.
For a map $g'_{1}:=\widetilde{k}_{n+1} \comp g_{1} : M \to X_{n+1}$, we obtain 
\begin{prop}\label{prop:g-1'-properties}
$\widetilde{K}_{n+1}$ is a smooth homotopy from $\id_{X_{n+1}}$ to $\widetilde{k}_{n+1}$ in $\smoothology{}$.
\end{prop}
\vskip-.5ex
\begin{Pf}
By definition, $\widetilde{K}_{n+1}(y,0)=y=\id_{X_{n+1}}(y)$ for all $y \in X_{n+1}$ and $\widetilde{K}_{n+1}(y,1)=\widetilde{k}_{n+1}(y)$ if $y \in D_{n+1}$.
Thus $\widetilde{K}_{n+1}$ is a smooth homotopy from $\id_{X_{n+1}}$ to $\widetilde{k}_{n+1}$. 
\end{Pf}

Thirdly, we choose an open covering $\mathcal{V}=\{B,A^{c}{=}M \smallsetminus A\}$ of $M$ and a smooth partition of unity $\rho_{B},\rho_{A^{c}} : M \to [0,1] \subset \real$ subordinate to $\mathcal{V}$, i.e., $\rho_{B}$ and $\rho_{A^{c}}$ are smooth functions satisfying $\supp{\rho_{B}} \subset B$, $\supp{\rho_{A^{c}}} \subset A^{c}$ and $\rho_{B}+\rho_{A^{c}}=1$ on $M$. 

We define a continuous map $H_{2} : M \times \real \to D_{n+1}$ by the following formula:\vspace{-.5ex}
$$
H_{2}(x,t) = \widetilde{K}_{n+1}(g_{1}(x),\rho_{A^{c}}(x){\cdot}t)
$$\vskip-.5ex\noindent
Then by definition, $H_{2}(x,0)=g_{1}(x)$ for $x \!\in\! M$.
Using a smooth function $\rho_{A^{c}}$, we obtain a map $g_{2} : M \to X_{n+1}$ given by $g_{2}(x)=H_{2}(x,1) = \widetilde{K}_{n+1}(g_{1}(x),\rho_{A^{c}}(x))$ for $x \!\in\! M$.
\begin{prop}\label{prop:g-2-properties}
$H_{2}$ gives a continuous homotopy relative to $A$ from $g_{1}$ to $g_{2}$ where $g_{2}$ is smooth on $B \cup (M \smallsetminus \supp{\rho_{V}})$ in $\smoothology{}$.
\end{prop}
\vskip-.5ex
\begin{Pf}
If $x \!\in\! A$, then $\rho_{A^{c}}(x)=0$ and $H_{2}(x,t)$ $=$ $\widetilde{K}_{n+1}(g_{1}(x),0)$ $=$ $g_{1}(x)$, for all $t \,{\in}\, \real$.
Thus $g_{1}$ is homotopic relative to $A$ to $g_{2}$:
By Proposition \ref{prop:g-1-properties} (\ref{prop:g-1-properties-2}) and (\ref{prop:g-1-properties-3}), $g_{1}$ is smooth on $B \cup (M \smallsetminus \supp{\rho_{V}})$ and $\widetilde{K}_{n+1}$ is smooth by Proposition \ref{prop:g-1'-properties}.
Hence $g_{2}$ is smooth on $B \cup (M \smallsetminus \supp{\rho_{V}})$, since $\rho_{A^{c}}$ is a smooth function.
\end{Pf}

Finally, we take open subsets $N=g_{1}^{-1}(V^{(-\eta)})$ and $L=g_{1}^{-1}(U)$ of $M$.
Here, $U \cup V^{(-\eta)} \supset X_{n+1}$ implies that $\{N,L\}$ is an open covering of $M$.
Moreover, $g_{2}\vert_{N} : N \to X_{n}$ is smooth on an open set $N \cap (B \cup (M \smallsetminus \supp{\rho_{V}}))$ in $\smoothology{}$ by Proposition \ref{prop:g-2-properties}, which is a superset of a closed set $N \cap (A \cup g_{1}^{-1}(\Cl{U}))$ in $N$ by Corollary \ref{cor:g-1-openset}.
By induction hypothesis, $g_{2}\vert_{N}$ is homotopic to a smooth map $g'_{N} : N \to X_{n}$ relative to $N \cap (A \cup g_{1}^{-1}(\Cl{U}))$ in $\smoothology{}$.

We define a map $g : M \to X_{n+1}$ by \vspace{-.5ex}
$$
g\vert_{N} = g'_{N} : N \to X_{n},\quad g\vert_{L} = g_{2}\vert_{L} : L \to X_{n+1}.
$$\vskip-.5ex\noindent
Since $g_{1}^{-1}(U) \subset g_{1}^{-1}(\Cl{U})$, it follows that $N \cap L = N \cap g_{1}^{-1}(U)$ is a subset of $N \cap (A \cup g_{1}^{-1}(\Cl{U}))$, and that $g'_{N}$ agrees with $g_{2}$ on $N \cap L$, which implies that $g$ is well-defined.
Since both of maps $g'_{N} : N \to X_{n}$ and $g_{2}\vert_{L} : L \to X_{n+1}$ are smooth in $\smoothology{}$, so is $g$. 
Furthermore, $g$ is continuously homotopic, by induction hypothesis, to $g_{2}$, by Proposition \ref{prop:g-2-properties}, to $g_{1}$, and, by Proposition \ref{prop:g-1-properties} (\ref{prop:g-1-properties-1}) to $f$.
It completes the proof of Theorem \ref{thm:main}.
\qed

\bigskip

\section*{Acknowledgements}

The author thanks Dan Christensen, Katsuhiko Kuribayashi, Kazuhisa Shimakawa, Tadayuki Haraguchi and Hiroshi Kihara for their kind and valuable comments and suggestions concerning on our recent study.
More precisely, they pointed out that a CW complex must be smooth around Theorem 9.7, Corollary 9.8 and entire \S 10 in \cite{MR3991180}, and Theorem A.1 should be proved rigorously, which is performed in this paper.

%
%

\bibliographystyle{alpha}
\bibliography{2020diff}

\begin{thebibliography}{CSW14}

\bibitem[Che73]{MR380859}
Kuo-tsai Chen.
\newblock Iterated integrals of differential forms and loop space homology.
\newblock {\em Ann. of Math. (2)}, 97:217--246, 1973.

\bibitem[Che75]{MR377960}
Kuo~Tsai Chen.
\newblock Iterated integrals, fundamental groups and covering spaces.
\newblock {\em Trans. Amer. Math. Soc.}, 206:83--98, 1975.

\bibitem[Che77]{MR454968}
Kuo~Tsai Chen.
\newblock Iterated path integrals.
\newblock {\em Bull. Amer. Math. Soc.}, 83(5):831--879, 1977.

\bibitem[Che86]{MR842915}
Kuo~Tsai Chen.
\newblock On differentiable spaces.
\newblock In {\em Categories in continuum physics ({B}uffalo, {N}.{Y}., 1982)},
  volume 1174 of {\em Lecture Notes in Math.}, pages 38--42. Springer, Berlin,
  1986.

\bibitem[CSW14]{MR3270173}
J.~Daniel Christensen, Gordon Sinnamon, and Enxin Wu.
\newblock The {$D$}-topology for diffeological spaces.
\newblock {\em Pacific J. Math.}, 272(1):87--110, 2014.

\bibitem[Har18]{arXiv:1811.06175}
Tadayuki Haraguchi.
\newblock Homotopy structures of smooth cw complexes.
\newblock {\em arXiv preprint arXiv:1811.06175}, 2018.

\bibitem[HS19]{arXiv:1912.05359}
Tadayuki Haraguchi and Kazuhisa Shimakawa.
\newblock On homotopy types of diffeological cell complexes.
\newblock {\em arXiv preprint arXiv:1912.05359}, 2019.

\bibitem[II19]{MR3991180}
Norio Iwase and Nobuyuki Izumida.
\newblock Mayer-{V}ietoris sequence for differentiable/diffeological spaces.
\newblock In {\em Algebraic topology and related topics}, Trends Math., pages
  123--151. Birkh\"{a}user/Springer, Singapore, 2019.

\bibitem[IZ13]{MR3025051}
Patrick Iglesias-Zemmour.
\newblock {\em Diffeology}, volume 185 of {\em Mathematical Surveys and
  Monographs}.
\newblock American Mathematical Society, Providence, RI, 2013.

\bibitem[Kih17]{arXiv:1702.04070}
Hiroshi Kihara.
\newblock Quillen equivalences between the model categories of smooth spaces,
  simplicial sets, and arc-gengerated spaces.
\newblock {\em arXiv preprint arXiv:1702.04070}, 2017.

\bibitem[Lee13]{MR2954043}
John~M. Lee.
\newblock {\em Introduction to smooth manifolds}, volume 218 of {\em Graduate
  Texts in Mathematics}.
\newblock Springer, New York, second edition, 2013.

\bibitem[ML98]{MR1712872}
Saunders Mac~Lane.
\newblock {\em Categories for the working mathematician}, volume~5 of {\em
  Graduate Texts in Mathematics}.
\newblock Springer-Verlag, New York, second edition, 1998.

\bibitem[Sou80]{MR607688}
J.-M. Souriau.
\newblock Groupes diff\'{e}rentiels.
\newblock In {\em Differential geometrical methods in mathematical physics
  ({P}roc. {C}onf., {A}ix-en-{P}rovence/{S}alamanca, 1979)}, volume 836 of {\em
  Lecture Notes in Math.}, pages 91--128. Springer, Berlin-New York, 1980.

\bibitem[SYH18]{MR3884529}
Kazuhisa Shimakawa, Kohei Yoshida, and Tadayuki Haraguchi.
\newblock Homology and cohomology via enriched bifunctors.
\newblock {\em Kyushu J. Math.}, 72(2):239--252, 2018.

\end{thebibliography}
\vskip1ex

\end{document}